\documentclass[a4paper,12pt]{amsart}
\usepackage[psamsfonts]{amssymb}
\usepackage{amscd,pifont}
\usepackage[usenames]{color}
\usepackage{amssymb,amsmath,amsthm}
\usepackage{euscript,eufrak,verbatim}
\usepackage{graphicx}
\usepackage[colorlinks,linkcolor=red,anchorcolor=blue,citecolor=blue]{hyperref}
\usepackage{amsfonts}
\usepackage{enumerate}
\usepackage{multicol}        % used for the two-column index
\usepackage[bottom]{footmisc}

\newtheorem{theo}{Theorem}[section]
\newtheorem{lem}[theo]{Lemma}
\newtheorem{prop}[theo]{Proposition}
\newtheorem{coro}[theo]{Corollary}
\theoremstyle{definition}

\newtheorem{exam}[theo]{Example}

\newtheorem*{theorem*}{Theorem}
\newtheorem*{conjecture*}{Conjecture}

\numberwithin{equation}{section}
\newcommand{\R}{{\mathbb R}}
\newcommand{\Z}{{\mathbb Z}}

\newcommand{\Q}{{\mathbb Q}}

\begin{document}
\baselineskip 18pt

\title{The structure of tiles in $\Z_{p^n}\times\Z_q$ and $\Z_{p^n}\times\Z_p$ }

\author%[authorlabel1]
{Shilei Fan}
\address%[authorlabel1]
{Shilei FAN: School of Mathematics and Statistics, and Key Lab NAA--MOE, Central China Normal University, Wuhan 430079, China}
\email{slfan@ccnu.edu.cn}

\author{Mamateli Kadir}
\address{Mamateli Kadir: School of Mathematics and Statistics \& Research Center of Modern Mathematics and Applications,
Kashi University, Kashi 844000, China}
\email{mamatili880@163.com}

\author{Peishan Li}
\address%[authorlabel1]
{Peishan Li: School of Mathematics and Statistics, and Hubei Key Lab--Math. Sci., Central China Normal University, Wuhan 430079, China}
\email{756298493@qq.com}

\thanks{S. L. FAN  and P. S. Li are  partially supported by NSFC (grants No. 12331004  and No.  12231013) and by NSF of Xinjiang Uygur Autonomous Region (Grant No. 2024D01A160). M. Kadir is supported by the NSF of China (Grant No. 12361015),
the Fundamental Research Founds for Colleges of XinJiang Education Department, China
(Grant No. XJEDU2023P108).}
\date{}%Dec. 15, 2011.}

%\dedicatory{}

%\commby{Daniel J. Rudolph}
%-----------------------------------------------------------

\begin{abstract} 
In this paper, we provide a geometric characterization of tiles in the finite abelian groups \( \mathbb{Z}_{p^n} \times \mathbb{Z}_q \) and \( \mathbb{Z}_{p^n} \times \mathbb{Z}_p \) using the concept of a \( p \)-homogeneous tree, which provides an intuitively visualizable criterion.
%In this paper, we provide a geometric characterization of tiles on the
%finite abelian groups $\Z_{p^n}\times\Z_q$ and $\Z_{p^n}\times\Z_p$  by the concept of $p$--homogeneous tree which is easily checkable.
 \end{abstract}
% -----------------------------------------------------------
\maketitle
%\tableofcontents
\section{Introduction}

\medskip

Let $G$ be a locally compact abelian group. Consider a Borel measurable subset $\Omega$ in $G$ with $0<\mathfrak{m}(\Omega)<\infty$, where $\mathfrak{m}$ denotes a  Haar measure on $G$. We say that  $\Omega$ is a
{\bf  tile}  of $G$ by translation if there exists a  set $T \subset G$
such that
\[\sum_{t\in T} 1_\Omega(x-t) =1, \mathfrak{m}\ a.e. x\in G,\] where $1_A$ denotes the indicator function of a set $A$ in $G$. Such a set $T$ is called a
{\bf tiling complement} of $\Omega$ and $(\Omega,T)$ is called a {\bf tiling pair}.

Let  $\widehat{G}$ be the dual group of $G$ consisting of the continuous characters of $G$. The set $\Omega$ is called a {\bf spectral set} if there exists a set $\Lambda\subset\widehat{G}$
of continuous characters of $G$ which form a Hilbert basis of the space $L^2(\Omega)$. Such a set $\Lambda$ is called a spectrum of $\Omega$ and $(\Omega,\Lambda)$ is called a {\bf spectral pair}.

In harmonic and functional analysis, a fundamental question asks whether a geometric property of sets (tiling) and an analytic property (spectrality) are always two sides of the same coin. This question was initially posed by Fuglede \cite{F74} for finite-dimensional Euclidean spaces, stemmed from a question of Segal on the commutativity of certain partial differential operators.

\begin{conjecture*}[Fuglede 1974]\label{Fuglede}
A Borel set $\Omega\subset \mathbb{R}^d$ of positive and finite Lebesgue measure is a spectral set if and only if it is a tile.
\end{conjecture*}

The original {\bf Fuglede conjecture} has been disproven in its full generality for dimensions 3 and above for both directions \cite{FMM06,KM2006,KM06,M05,T04}. This means that neither implication (tiling implies spectral and vice versa) holds true in these higher dimensions.  However, the connection between tiling and spectral properties remains an active area of research, particularly in lower dimensions. The conjecture is still open for the one-dimensional and two-dimensional cases ($\R$ and $\R^2$). There might be a deeper relationship to be discovered in these simpler settings (see \cite{DL2014} for a focused look on dimension $1$).

Despite the general counterexamples, the conjecture has been successfully proven for convex sets  \cite{IKT03,GL17,LM21}  in all dimensions.
On the other hand, there has been a growing interest in extending the Fuglede conjecture beyond the realm of Euclidean spaces. Fuglede himself hinted at the possibility of exploring the conjecture in different settings. This has led to a more general version of the conjecture applicable to locally compact abelian groups.
 Formally stated, Fuglede's conjecture on a locally compact abelian group $G$ asks:
 {\bf Is a Borel set $\Omega\subset G$ of positive and finite Haar measure is a spectral set if and only if it is a tile?}
% Iosevich, Katz, and Tao  \cite{IKT03}  initiated this progress in 2003 by demonstrating the validity of the conjecture for $d = 2$ (two dimensions). This result was later extended to $d = 3$ (three dimensions) by Greenfeld and Lev \cite{GL17} in 2017. Finally, Lev and Matolcsi \cite{LM21} achieved a major breakthrough in 2021 by proving the conjecture for all dimensions (general n) under the additional assumption that the sets are convex.

The generalized Fuglede conjecture has been proved for different groups, particularly for finite abelian groups.  These successes include groups like $\mathbb{Z}_{p^{n}}$ \cite{L02}, $\mathbb{Z}_{p}^{d}$ ($p=2$ and $d\le6$; $p$ is an odd prime and $d=2$; $p=3,5,7$ and $d=3$) \cite{AABF17,FMV2022,FS20,IMP17},  $\mathbb{Z}_{p}\times\mathbb{Z}_{p^{n}}$ \cite{IMP17,S20,Z2022},    $\mathbb{Z}_{p}\times\mathbb{Z}_{pq}$ \cite{KS2021}  and $\mathbb{Z}_{pq}\times\mathbb{Z}_{pq}$ \cite{FKS2012},  $\mathbb{Z}_{p^{n}q^{m}}$ ($p<q$ and $m\le9$ or $n\le6$; $p^{m-2}<q^{4}$) \cite{KMSV20,M21,MK17}, $\mathbb{Z}_{pqr}$ \cite{S19}, $\mathbb{Z}_{p^{2}qr}$ \cite{Somlai21} and $\mathbb{Z}_{pqrs}$ \cite{KMSV2012}, where $p,q,r,s$ are distinct primes. Fan et al. \cite{FFLS19,FFS16} established the validity of the conjecture for the field $\Q_p$ of $p$-adic numbers, and  obtained the geometric structure of tiles and spectral sets in $\Q_p$. The  concept of $p$--homogeneous tree structure was introduced in \cite{FFS16} to characterize the tile in  $\mathbb{Z}_{p^{n}}$.

It is known that spectral conjecture holds on  $\Z_{p^n}\times\Z_q$ and $\Z_{p^n}\times\Z_p$. However, the geometric structures of tiles were not mentioned. In this paper, we provide a geometric characterization of tiles in the
finite abelian groups $\Z_{p^n}\times\Z_q$ and $\Z_{p^n}\times\Z_p$  by the concept of $p$-homogeneous tree.

 Firstly, we give a quick recall of the concept  $p$-homogeneous set in the cyclic group  $\Z_{p^n}$. Consider the group $\Z_{p^n}=\{0, 1, \cdots, p^n-1\}$ as a finite tree $\mathcal{T}^{(n )}$,
where $n$ is an integer (see Figure \ref{fig:1}). The vertices of $\mathcal{T}^{(n )}$ are the sets of $\Z_{p^\gamma}, \ 0\leq \gamma\leq n$
and translations of them. The point of $\Z_p$ corresponding to the point of vertices at level $\gamma$.
The set of edges consists of pairs $(x, y)\in \Z_{p^\gamma}\times \mathbb{Z}_{p^{\gamma+1}}$ such that
$x \equiv y (\bmod p^{\gamma})$, where $0 \leq \gamma \leq n-1$.
Each subset $C \subset \Z_{p^n}$ will determine a subtree of
$\mathcal{T}^{(n )}$, denoted by $\mathcal{T}_{C}$, which consists of the paths from the root to the boundary points
in $C$.
%For any $0 \leq \gamma \leq n$, we set
%$$
%{C_{\bmod {p^\gamma}}}:=\{x\in \{0, 1, 2, \dots, p^{\gamma}-1\}:\exists\: y \in C,~ \text{such that}~ x \equiv y \: \bmod \:p^{\gamma}\}.
%$$
%The set $C_{\bmod{p^\gamma}}$ is determined by $C$ modulo $p^{\gamma}$, and subset of $\Z_{p^\gamma}$.
%
%The vertices of $\mathcal{T}_C$ consists of the disjoint union of the sets ${C_{\bmod {p^{\gamma}}}}$~, $0 \leq \gamma \leq n$.
%Where ${C_{\bmod {p^{\gamma}}}}$ stands for the whole points on $\gamma$-th level of the tree $\mathcal{T}_C$.
%The set of edges consists of pairs $(x,\ y)\in {C_{\bmod {p^{\gamma}}}} \times {C_{\bmod {p^{\gamma + 1}}}}$ such that
%$x \equiv y (\bmod p^{\gamma})$, $0 \leq \gamma \leq n-1$. For the vertices $s \in C_{\bmod {p^{\gamma}}}$ and $u\in C_{\bmod {p^{\gamma+1}}}$,
%we call $u$ a descendant of $s$, and $s$ the parent of $u$, if there is an edge between $s$ and $u$.

 \begin{figure}[h]
  	\centering
  	\includegraphics[width=0.7\linewidth]{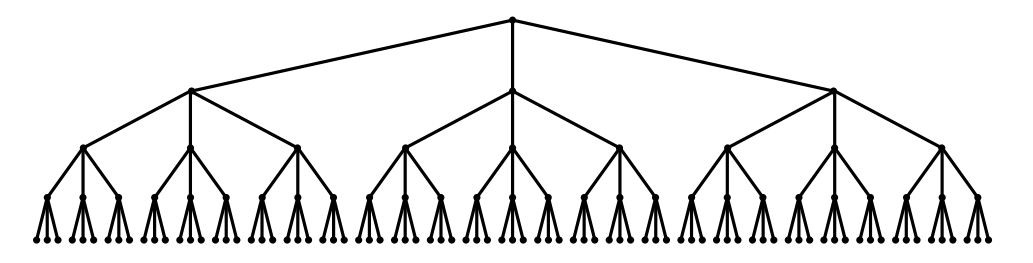}
  	\caption{Consider the set $\Z_{3^4}=\{0, 1, 2, \cdots, 80\}$ as a tree $\mathcal{T}^{\left(4\right)} $. }
  	\label{fig:1}
  \end{figure}

 \begin{figure}[h]
  	\centering
  	\includegraphics[width=0.7\linewidth]{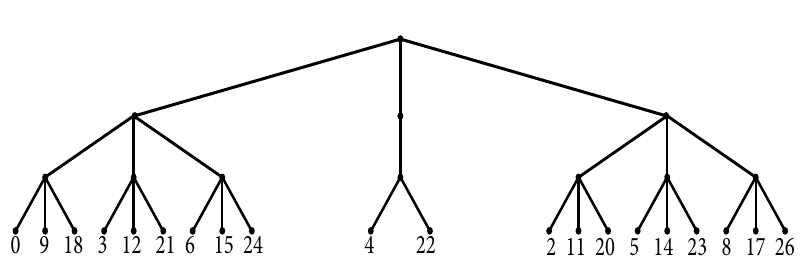}
  	\caption{The subtree $\mathcal{T}_{C}$ of  $\mathcal{T}^{\left(3\right)} $ with $C=\{0,2,3,4,\cdots, 24,26\}$. }
  	\label{fig:sub}
  \end{figure}

Let $I \subseteq \{0, 1, 2, \dots, n -1\}$, $J\in \{0, 1, 2, \dots, n -1\}\backslash I$. We say a
subtree $\mathcal{T}_C$ of $\mathcal{T}^{(n)}$ is a $\mathcal{T}_I$-form, if each point on the $I$-th levels of the tree $\mathcal{T}_C$
has $p$ descendants, and the each point on the $J$-th levels only has one descendant.
A $\mathcal{T}_I$-form tree is called a finite {\em $p$-homogeneous tree}. We call a set $C \subset \Z_{p^n}$
a {\em $p$-homogeneous set}, if the corresponding tree $\mathcal{T}_C$ is a {\em $p$-homogeneous tree}.
A special subtree $\mathcal{T}_I$ is shown in figure \ref{fig:2}.
  \begin{figure}[h]
  	\centering
  	\includegraphics[width=0.7\linewidth]{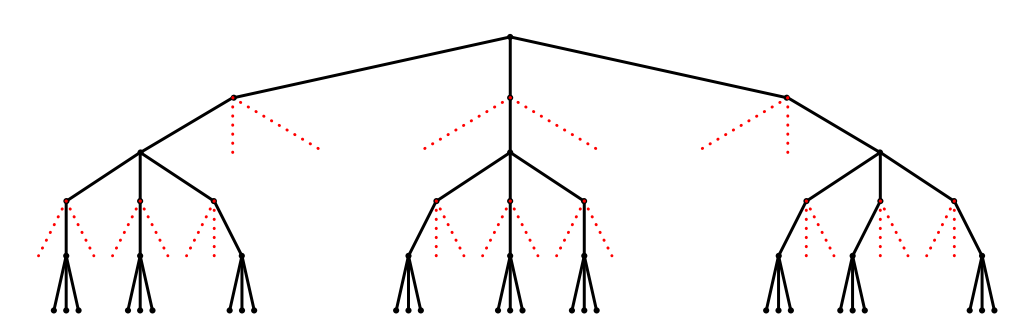}
  	\caption{For $p=3$, a ${{\mathcal T}_I}$-form tree with $n=5$, $I=\{0,2,4\}$.}
  	\label{fig:2}
  \end{figure}

%Before stating our main results, we fix some notation which will be used throughout this paper.
%\begin{itemize}
%\item $\Z_N=\{0, \cdots, N-1\}$. The ring of integers {\em modulo} $N$.
%%\item $\Z_N^\ast=\{a\in \Z_N: \gcd(a, N)=1\}$. It is the group of units of $\Z_N$.
%%\item $[a, b]=\{a, a+1, \cdots, b\}$. Where $a, b$ are two integers with $a\leq b$.
%\item $1_A$: the characteristic function of a set $A$.
%\item $|S|$: the cardinality of a set $S\subset G$.
%\end{itemize}

\subsection{The structure of tiles in  $\Z_{p^n}\times\Z_q$}
In \cite{CM99}, Coven and Meyerowitz introduced the so called CM conditions, which play a crucial role in characterizing  a set $A$ that tiles the finite cyclic group  $\Z_N$ by
translations.

Let $A \subseteq \Z_N$ be a multi-set, and let $m_a$ denote the multiplicity of $a \in A$. The \emph{mask polynomial} of $A$ is defined as
\[
A(x) = \sum_{a \in A} m_a x^a.
\]
Denote by $\omega_N=e^{2\pi i /N}$,  which is a primitive $N$-th root of unity. %For any $d\in \{0,\cdots  N-1\}$, it is straightforward to verify that $A(\omega^d) = 0$ is equivalent to $\widehat{1_A}(d) = 0$.

Denote by $\Phi_s(x)$  the $s$-th cyclotomic polynomial. Let $S$ denote the set of prime powers dividing $N$, and define
\[
S_A = \{s \in S : \Phi_s(x) \mid A(x)\}.
\]
Consider the following algebraic conditions
\begin{itemize}
\item[(T1)] $|A| = A(1) = \prod_{s \in S_A} \Phi_s(1)$.
\item[(T2)] Let   $s_1, \dots, s_m \in S_A$ be powers  of different primes. Then the polynomial $\Phi_{s_1 \dotsm s_m}(x)$ divides $A(x)$.
\end{itemize}
 It is proved in \cite{CM99} that  if $A$ satisfies properties \textbf{\rm (T1)} and \textbf{\rm (T2)}, then $A$ tiles $\Z_N$ by translations.

Write  the cyclic group $\Z_{p^nq}$ as the product form $\Z_{p^n}\times\Z_q$. We give a geometric characterization of tiles in  $\Z_{p^n}\times\Z_q$ by their $p$-homogeneous structure. For a set $E$, denote by $|E|$ the cardinality  of $E$.

\begin{theo}\label{main1}
 Let $ \Omega=\bigsqcup\limits_{j=0}^{q-1}(  \Omega_{j}  \times \left \{ j \right \})$ be a tile in $\Z_{p^n}\times\Z_q$
 with $ \Omega_j\subset \Z_{p^n}$.
 Then we have the following two cases.
\begin{enumerate}[{\rm(1)}]
  \item  If $|\Omega|=p^m$, then $ \Omega_0,  \Omega_1 ,\dots ,  \Omega_{q-1}$ are disjoint, and $\bigsqcup\limits_{j=0}^{q-1}  \Omega_{j}$ is $p$-homogeneous.
  %corresponds ${\mathcal{T}_{I,J}}$-tree, where $ I =\{n - {a_m} - 1,n - {a_{m - 1}} - 1, \cdot  \cdot  \cdot, n - {a_1} - 1\},
%  J = \{0, 1, \cdot  \cdot  \cdot , n - 1\} \backslash I .$ On the other hand, there exists
%$T \subseteq\Z_{p^n}$, such that $\mathcal{T}_T$ is a $\mathcal{T}_{J,\ I}$ tree, and $(\bigsqcup\limits_{j=0}^{q-1} C_{j}, T)$
%is a tiling pair in $\Z_{p^n}$.

  \item If $|\Omega| = {p^m}q$, then $|{ \Omega_i}|={p^m}$ for any $i$, and all $ \Omega_i$ are  $p$-homogeneous with a common branched level set.
  %where $ I =\{{a_m} ,{a_{m - 1}}, \cdot  \cdot  \cdot ,{a_1} \} $, $J = \{0, 1, \cdot  \cdot  \cdot , n - 1\} \backslash I $. And there is a set
%  $T \subseteq  \Z_{p^n}$, such that $\mathcal{T}_T$ is a $\mathcal{T}_{J,\ I}$ tree, and $(C_i, T)$ is a tiling pair in $\Z_{p^n}$ for any $i$.
\end{enumerate}
\end{theo}

\subsection{The structure of tiles in  $\Z_{p^n}\times\Z_p$}
Now we shall investigate the geometric structure of tiles in the group $\Z_{p^n}\times\Z_p$. Let $ \Omega$ be a tile of $\Z_{p^n}\times\Z_p$.
It is evident that $|\Omega|$ divide $p^{n+1}$. Hence, assume that  $| \Omega|=p^i$ for some $1 \leq i\leq n$.
%Through out the paper, we assume that $|A|=p^t$ for $2 \leq t \leq n$.
 Define a map $\pi_1$ from the group $ \Z_{p^n}\times\Z_p$ to the group  $\Z_{p^n}$ by

\[ \pi_1(a, b)=a, \quad  \hbox{ for } (a, b)\in \Z_{p^n}\times\Z_p. \]

 Let
 \[\mathcal{Z}_ \Omega=\Big\{g\in \Z_{p^n}\times\Z_p: \widehat {1_ \Omega}(g)=0 \Big\}\] be the set of zeros of the Fourier
transform of the function $1_\Omega$.
We have the following characterization theorem of tiles in $\Z_{p^n}\times\Z_p$.

\begin{theo}\label{main2}
 Assume  $(\Omega,T)$ is  a tiling pair in $\Z_{p^n}\times\Z_p$ with $| \Omega|=p^t$.
Let
$$
\mathcal{I}_{\Omega}=\big\{i\in \{0,\cdots \ n-1\}: (p^i, 0)\in \mathcal{Z}_ \Omega \big\}.
$$

%$$
% C_I = \left\{(p^i,\alpha):i \in I,\ \alpha \in [0,\ p-1] \right\},$$
% $$C_J = \left\{(p^j,\alpha): j \in J,\ \alpha \in [0,\ p-1] \right\}.
% $$
We distinguish three cases.
\begin{enumerate}[{\rm(1)}]
\item If $| \mathcal{I}_{\Omega}|=t$, then the set $\pi_1( \Omega)$ is a $p$-{\it homogeneous}  in  $\Z_{p^n}$ with $|\pi_1( \Omega)|=p^t.$
\item If $| \mathcal{I}_{\Omega}|=t-1$ and    $ (p^j, b)\notin \mathcal{Z}_\Omega$  for each  $ j\in  \{0,1,\cdots, n\}\setminus \mathcal{I}_{\Omega}$ and $b\in \Z_p^{*}$,  then  the sets
\[ \Omega_i=\{x\in \Z_{p^n}: (x, i) \in \Omega\}\]  are  $p$-homogenous in $\Z_{p^n}$ with  a same branched level set and  $|\Omega_i|=p^{t-1}$.
 \item If $| \mathcal{I}_{\Omega}|=t-1$ and   $(p^{j}, b) \in \mathcal{Z}_{\Omega}$ for some  $j\in   \{0,1,\cdots, n\}\setminus \mathcal{I}_{\Omega}$ and $b\in \Z_p^*$, then
 the set  \[\widetilde{\Omega} = \{ x +b_0 y p^{n-j_0-1}: (x, y) \in   \Omega\} \] is  $p$-homogeneous in $\Z_{p^n}$  with
 $| \widetilde{\Omega}|=p^t$,  where   $j_0$ is the minimal  number in
 $\{0,1,\cdots, n\}\setminus\mathcal{I}_{\Omega}$ such that   $(p^{j_0},b_0)\in \mathcal{Z}_\Omega $ for some  $ b_0\in \Z_p^*$.
  For all $(x,y)\in \Omega$, the sets \[\Omega_{x,y}:=\{x^{\prime }\in\Z_{p^n}: (x^{\prime},y)\in \Omega  \text{ and } x^{\prime}\equiv x \mod p^{n-j_0-1}  \
\}\]  are  $p$-homogenous with  a same branched level set.
\end{enumerate}
\end{theo}

\medskip

\section {Preliminaries}

\medskip

In this section, we present preliminaries for the proof of the main
results. We start with the recall of the Fourier transform on finite abelian groups,
basic properties of spectral sets and tiles on finite abelian groups, $p$-homogeneous sets and $\Z$-module generated by the $p^n$-th roots of unity.

Let $G$ be a finite abelian group, and let $\mathbb{C}$ be the set of complex numbers. A character on $G$ is a group homomorphism $\chi: G \to \mathbb{C}$. The dual group of a finite abelian group $G$, denoted as $\widehat{G}$, is the character group of $G$.

 %In other words, it is the set of all group homomorphisms (characters) from $G$ to the complex numbers under multiplication.

For a finite abelian group $G$, it can be written as  $\Z_{n_1} \times \Z_{n_2} \times \cdots \times \Z_{n_s}$ with $n_1\mid n_2 \mid \cdots \mid n_s$.  For $g=(g_1,g_2, \dots, g_s)\in G$, denote by $\chi_g$ the character
\[
\chi_{g}(x_1, \dots, x_s) = e^{2\pi i \sum_{j=1}^{s} \frac{x_jg_j}{n_j}}.
\]
For $g,h\in G$, it is clear that
\[\chi_{g+h}(x)=\chi_{g}(x)\cdot \chi_{h}(x).\]
The dual group  $\widehat{G}$ is isomorphic to itself, i.e.
\[\widehat{G}=\widehat{\Z}_{n_1} \times \widehat{\Z}_{n_2} \times \cdots \times \widehat{\Z}_{n_s}\cong\Z_{n_1} \times \Z_{n_2} \times \cdots \times \Z_{n_s}.\]

For  two  finite abelian groups $G_1,G_2$, let $G$ be their  product $G_1\times G_2$. It is known that
$\widehat{G}=\widehat{G_1}\times \widehat{G_2}\cong G_1\times G_2$, and each character in $\widehat{G}$ can be written as \[\chi_{(g_1,g_2)}(x_1,x_2)=\chi_{g_1}(x_1)\chi_{g_2}(x_2),\]
where $g_1\in G_1$ and $g_2\in G_2$.

\subsection{Fourier transform on the finite groups.}
The Fourier transform on $G$ is a linear transformation that maps a function $f: G \rightarrow \mathbb{C}$  to a function $\widehat{f}: \widehat{G} \rightarrow \mathbb{C}$ defined as follows:
\[\widehat{f}(g) = \sum_{x \in G} f(x) \cdot \chi_g(x)\]
where $\chi_g$ is the character of $G$  corresponding to $g$, and $f(x)$ is the value of the function $f$ at the element $x$ in $G$.

For $A \subset G$, denote by
\[\mathcal{Z}_{A}:=\big\{x\in \widehat{G}: \widehat{1_{A}}(x)=0\big\}\]
the set of zeros  of the Fourier transform of the indicator function $1_A$. It is clear that $\mathcal{Z}_A$ is invariant under translation.
\begin{lem}\label{lem-traninv}
For each $g \in G$, we have \[  \mathcal{Z}_{A}=\mathcal{Z}_{A+g}.\]
\end{lem}
\begin{proof}
Assume that $\xi \in \mathcal{Z}_A$ which is equivalent to $\widehat{1_{A}}(\xi)=\sum_{x\in A} \chi_{\xi}(x)=0.$
Hence, \[\widehat{1_{A+g}}(\xi)=\sum_{x\in A+g} \chi_{\xi}(x)=\sum_{x\in A} \chi_{\xi}(x+g)=\chi_{\xi}(g)\sum_{x\in A} \chi_{\xi}(x)=0.\]

\end{proof}

The following lemma states that the set of zeros of the Fourier transform of the indicator function of
a set is invariant  under special multiplication.

\begin{lem}[\cite{Z2022}]\label{unit}
If $g\in \mathcal{Z}_A $, then  $rg \in \mathcal{Z}_A $ for any integer $r$  with $\gcd(r,|G|)=1$.
\end{lem}

%For any finite abelian group $G$, by theorem \ref{fundamental}, we have
%\[
%G\cong \Z_{N_1}\times \Z_{N_2}\times \cdots \times\Z_{N_s}.
%\]
%Hence, we can define inner product on $G$ as
%\[\langle x,y \rangle=\sum_{i=1}^s k_i n_i\frac{N_s}{N_i}, \quad \forall x=(n_1, \cdots, n_s),\ y=(k_1, k_2,\cdots, k_s)\in G. \]

 Let $G$ be  the product of  finite abelian groups $H, S$. Let  $A\subset G=H\times S$. For each $s\in S$, define
 \[ A_s=\big\{h\in H: (h,s)\in A\big\}.\]
\begin{lem}\label{lem:periodic}
If $(h, s) \in \mathcal{Z}_A$ for each
$s \in  S$, then  $h\in  \mathcal{Z}_{A_s}$ for each
$s \in  S$.
 \end{lem}
 \begin{proof}
Note that \[A=\bigcup_{s\in S}A_s\times\{s\}.\]
Since for each $s \in S$,  $(h,s)\in \mathcal{Z}_A$, it follows that
\begin{align} \label{eq:zero1}
\sum_{(x,y)\in A} \chi_{(h,s)}(x,y)=\sum_{s\in S}\chi_{s}(y) \sum_{x\in A_s} \chi_{h}(x)=0.
\end{align}
Let $X_s= \sum_{x\in A_s} \chi_{h}(x)$ and   let $\mathbf{X}=(X_s)_{s\in S}$ be the row vector with elements $X_s$.
Let   $M_S=(\chi_{\alpha}(s))_{\alpha,s \in S}$  be the Fourier matrix  of $S$. By Equality \eqref{eq:zero1}, we obtain the following system of linear equations

\[
%\begin{pmatrix}
%1 & 1 & \cdots & 1 \\
%1 & e^{2pi \frac{1}{p} }& \cdots &  e^{2pi \frac{p-1}{p} } \\
%\vdots & \vdots & \ddots & \vdots \\
%1 &  e^{2pi \frac{p-1}{p}} & \cdots & e^{2pi \frac{(p-1)^2}{p}}
%\end{pmatrix}
M_S \cdot  \mathbf{X}
%\begin{pmatrix}
%X_0 \\
%X_h \\
%\vdots \\
%xx
%\end{pmatrix}
=
\begin{pmatrix}
0 \\
0 \\
\vdots \\
0
\end{pmatrix}.
\]
The coefficient matrix $M_S$ is of full rank, therefore $X_s=0$ for each $ s\in S$, which implies that
\[\widehat{1_{A_s}}(h)= \sum_{x\in A_s} \chi_{h}(x)=X_s=0.\]
 \end{proof}
\subsection{Basic properties of  tiles.}
In this section, we give some properties of  tiles in finite abelian groups. Considering a tiling pair, we  have the following equivalent characterization.
%We restate the definition of the spectral set and tile on the finite abelian groups.

%Let $G$ be a finite abelian group.  For a subset $A$ of  $G$,
%define
% \[
%\mathcal{Z}_A=\Big\{g\in G: \widehat {1_A}(g)=0 \Big\}
%\]
%be the set of zero of the Fourier transform  of the indicator function $1_{A}$.

%Recall that $ \Omega\subseteq G$ is called a {\em spectral set} if there exists a set $\Lambda \subset G$ such that
% \[\left\{ {{\chi _\lambda}:\lambda \in \Lambda} \right\}\] forms an orthogonal basis of the space $L^2(\Omega)$.
%Such a set $\Lambda$ is called a {\em spectrum} of $\Omega$ and $(\Omega, \Lambda)$ is called a {\em spectral pair}.
%%
%% Recall that
%%$$
%%\mathcal{Z}_\Omega=\Big\{g\in G: \widehat {1_\Omega}(g)=0 \Big\}.
%%$$
%%And
%%
%%$$
%%\Z_N^*=\Big\{a \in \Z_N: \gcd(a, N)= 1\Big\}.
%%$$
%\begin{lem}[\cite{Z2022}]
%Let $\Omega, \Lambda \subseteq G $. Then the following statements are equivalent:
%\begin{enumerate}[{\rm(1)}]
%  \item $(\Omega, \Lambda)$ is a spectral pair.
%  \item $(\Lambda,\Omega)$ is a spectral pair.
%  \item $|\Omega|=|\Lambda|$ and $(\Lambda-\Lambda)\backslash \{0\}\subseteq \mathcal{Z}_\Omega.$
%\end{enumerate}
%\label{lem2.4}
%\end{lem}

%\begin{defi}
%Let $G$ be a finite abelian group. We call $\Omega\subseteq G$ a {\em tile} if there exists a set $T\subseteq G$
%such that each element $g\in G$ can be uniquely expressed as \[g = a + t,~a \in \Omega,~t \in T.\]
%In this case, we denote by $G = \Omega \oplus T$. The set $T$ is called a {\em tiling set} of $\Omega$ and $(\Omega, T)$
%is called a {\em tiling pair}.
%\end{defi}
\begin{lem}[\cite{S20}]
Let $\Omega, T \subseteq G $. Then the following statements are equivalent:
\begin{enumerate} [{\rm(1)}]
  \item $(\Omega, T)$ is a tiling pair.
  \item $(T, \Omega)$ is a tiling pair.
  \item $(\Omega+x, T+y)$ is a tiling pair for any $x, y\in G$.
  \item $|\Omega|\cdot|T|=|G|$ and $(\Omega-\Omega)\cap(T-T)= \{0\}.$
  \item $|\Omega|\cdot|T|=|G|$ and $\mathcal{Z}_\Omega\cup\mathcal{Z}_T= G\setminus\{0\}.$
\end{enumerate}
\label{2.6}
\end{lem}

\begin{lem}[{\cite[Theorem 3.17]{SS09book}}\label{lemma-replace}]
Assume that   $(\Omega,T)$  is a tiling pair of  a finite abelian group $G$.
For any integer $k$  with $(k,|T|)$=1,   $(\Omega,kT)$ is also a tiling pair.
\end{lem}

\subsection{$\Z$-module generated by $p^n$-th roots of unity.}

Let $m\geq2$ be an integer and $\omega_m=e^{2\pi i/m}$, which is a primitive
$m$-th root of unity. Denote
\begin{equation*}
\mathcal{M}_m=\Big\{(a_0, a_1, \cdots, a_{m-1})\in\mathbb{Z}^m:\ \sum_{j=0}^{m-1}a_j\omega_m^j=0\Big\},
\end{equation*}
which form a $\mathbb{Z}$-module. In the following we assume that $m=p^n$ is a prime power.

\begin{lem}[\cite{FFS16}]\label{lem1}
Let $(a_0, a_1, \cdots, a_{p^n-1})\in\mathcal{M}_{p^n}$.
Then for any integer $0\leq k\leq p^{n-1}-1$, we have $a_k=a_{k+jp^{n-1}}$ for all $j=0, 1, \cdots, p-1$.
\end{lem}
We will use Lemma \ref{lem1} in the following two particular forms. The first one is an immediate consequence.

\begin{lem}[\cite{FFS16}]\label{2}
Let $(b_0,b_1,\cdots ,b_{p-1})\in\mathcal{M}_p$. Then
subject to a permutation of $(b_0, b_1, \cdots, b_{p-1})$, there exist $0\leq r\leq p^{n-1}-1$, such that
$$b_j\equiv r+jp^{n-1}\bmod p^n$$ for all $j=0, 1, \cdots, p-1$.
\end{lem}

\begin{lem}[\cite{FFS16}]\label{fenzhi}
 Let $C$ be a finite subset of $\mathbb{Z}$. If $\sum_{c\in C}e^{2\pi ic/{p^n}}=0$, then $p\mid |C|$
 and $C$ can be decomposed into $|C |/p$ disjoint subsets $C_1,C_2,\cdots, C_{|C| /p}$, such that
 each subset consists of $p$ points and $$\sum_{c\in C_j}e^{2\pi ic/{p^n}}=0.$$
\end{lem}

\subsection{ $p$-homogeneity  of tiles in $\Z_{p^n}$}

Let $n$ be a positive integer. To any finite sequence $t_0 t_1 \cdots t_{n-1} \in \{0, 1, \dots, p-1\}^{n}$, we associate the integer
\[
c = c(t_0 t_1 \cdots t_{n-1}) = \sum_{i=0}^{n-1} t_i p^i \in \{0, 1, \dots, p^n - 1\}.
\]
This establishes a bijection between $\Z_{p^n}$ and $\{0, 1, \dots, p-1\}^{n}$, which we consider as a finite tree, denoted by ${\mathcal T}^{(n)}$ (see Figure~\ref{fig:1}).

The set of vertices of ${\mathcal T}^{(n)}$ is the disjoint union of the sets $\Z_{p^\gamma}$, for $0 \le \gamma \le n$. Each vertex, except the root, is identified with a
sequence $t_0 t_1 \cdots t_{\gamma-1}$ in $\Z_{p^\gamma}$, where $0 \le \gamma \le n$ and $t_i \in \{0, 1, \dots, p-1\}$. The set of edges consists of pairs $(x, y) \in \Z_{p^\gamma}
\times \Z_{p^{\gamma+1}}$ such that $x \equiv y \pmod{p^\gamma}$, where $0 \le \gamma \le n-1$. The points in $\Z_{p^\gamma}$ are corresponding to the vertices at level $\gamma$.
Each point $c$ of $\Z_{p^n}$ is identified with a boundary  vertex $\sum_{i=0}^{n-1} t_i p^i \in \{0, 1, \dots, p^n - 1\}$ which is located at level $n$.

Each subset $C \subset \Z_{p^n}$ determines a subtree of ${\mathcal T}^{(n)}$, denoted by ${\mathcal T}_C$, which consists of the paths from the root to the boundary points in $C$. For each $0 \le \gamma \le n$, we denote by
\[
C_{\bmod p^\gamma} := \big\{x \in \{0, 1, \dots, p^\gamma-1\} : \exists\ y \in C \text{ such that } x \equiv y \pmod{p^\gamma}\big\}
\]
the subset of $C$ modulo $p^\gamma$.

The set of vertices of ${\mathcal T}_C$ is the disjoint union of the sets $C_{\bmod p^\gamma}$, for $0 \le \gamma \le n$. The set of edges consists of pairs $(x, y) \in C_{\bmod p^\gamma} \times C_{\bmod p^{\gamma+1}}$ such that $x \equiv y \pmod{p^\gamma}$, where $0 \le \gamma \le n-1$.

For vertices $u \in C_{\bmod p^{\gamma+1}}$ and $s \in C_{\bmod p^\gamma}$, we call $s$ the  \emph{parent} of $u$ or $u$ a \emph{descendant} of $s$ if there exists an edge between $s$ and $u$.

Now, we proceed to construct a class of subtrees of ${\mathcal T}^{(n)}$. Let $I$ be a subset of $\{0, 1, \dots, n-1\}$, and let $J$ be its complement. Thus, $I$ and $J$ form a partition of $\{0, 1, \dots, n-1\}$, and either set may be empty.

We say a subtree ${\mathcal T}_C$ of ${\mathcal T}^{(n)}$ is of ${\mathcal T}_{I}$-form if its vertices satisfy the following conditions:

\begin{enumerate}
    \item If $i \in I$ and  $t_0 t_1\dots t_{i-1} $ is given, then $t_i$ can take any value in $\{0, 1, \dots, p-1\}$. In other words, every vertex in $C_{\bmod p^{i-1}}$
        has $p$ descendants.
    \item If $i \in J$ and $t_0 t_1 \dots t_{i-1}$ is given, we fix a value in $\{0, 1, \dots, p-1\}$ that $t_i$ must take. That is, $t_i$ takes only one value from $\{0, 1, \cdots, p-1\}$, which depends on $t_0 t_1 \dots t_{i-1}$. In other words, every vertex in $C_{\bmod p^{i-1}}$ has one descendant.
\end{enumerate}
%  \begin{figure}
%  	\centering
%  	\includegraphics[width=0.7\linewidth]{1}
%  	\caption{The set $\Z_{3^4}=\{0, 1, 2, \cdots, 80\}$ is considered as a tree ${\mathcal T}^{\left(4\right)} $. }
%  	\label{fig:1}
%  \end{figure}

Note that such a subtree depends not only on $I$ and $J$ but also on the specific values assigned to $t_i$ for $i \in J$. A ${\mathcal T}_{I}$-form tree is called a finite \emph{$p$-homogeneous tree}, see Figure~\ref{fig:7} for an example.

A set $C \subset \Z_{p^n}$ is said to be  \emph{$p$-homogeneous} subset of $\Z_{p^n}$ with  \emph{branched level set} $I$ if the corresponding tree ${\mathcal T}_C$ is $p$-homogeneous of form  ${\mathcal T}_{I}$.

\begin{figure}[h]
  	\centering
  	\includegraphics[width=0.7\linewidth]{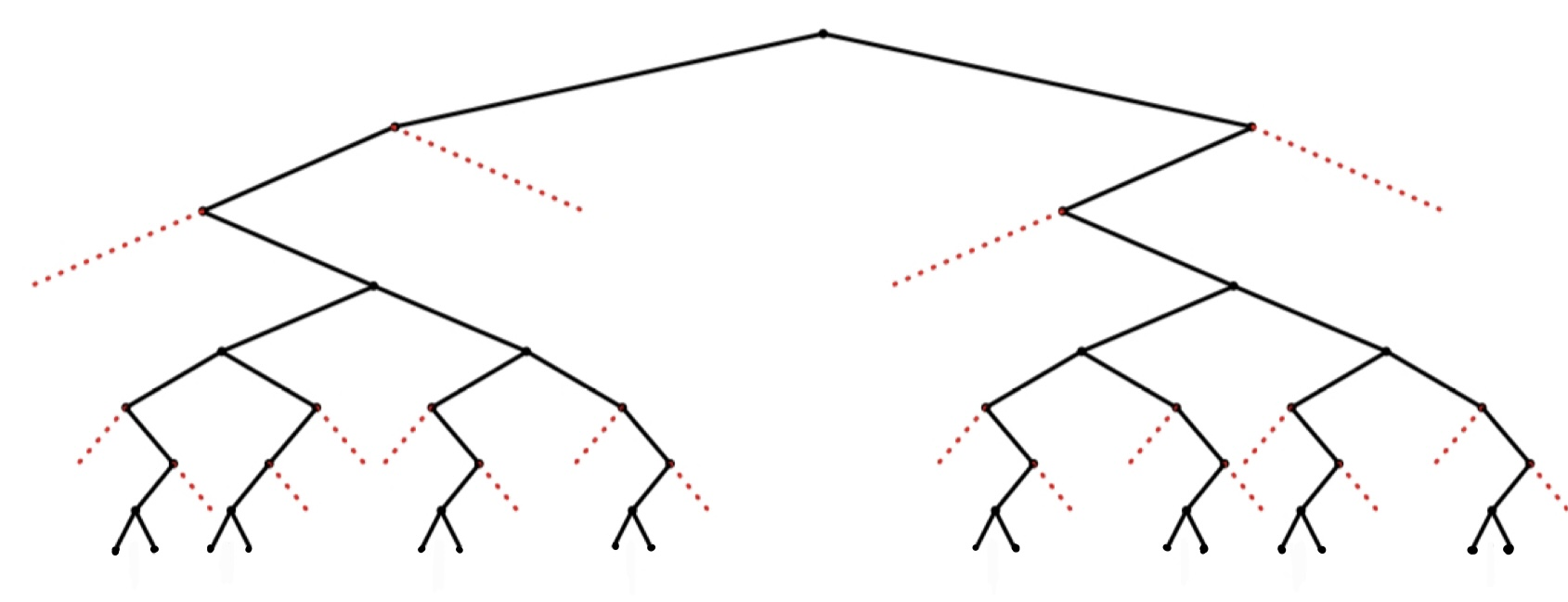}
  	\caption{For $p=2$, a $p$-{\it homogeneous} tree.}
  	\label{fig:7}
  \end{figure}

\begin{figure}[h]
  	\centering
  	\includegraphics[width=0.9\linewidth]{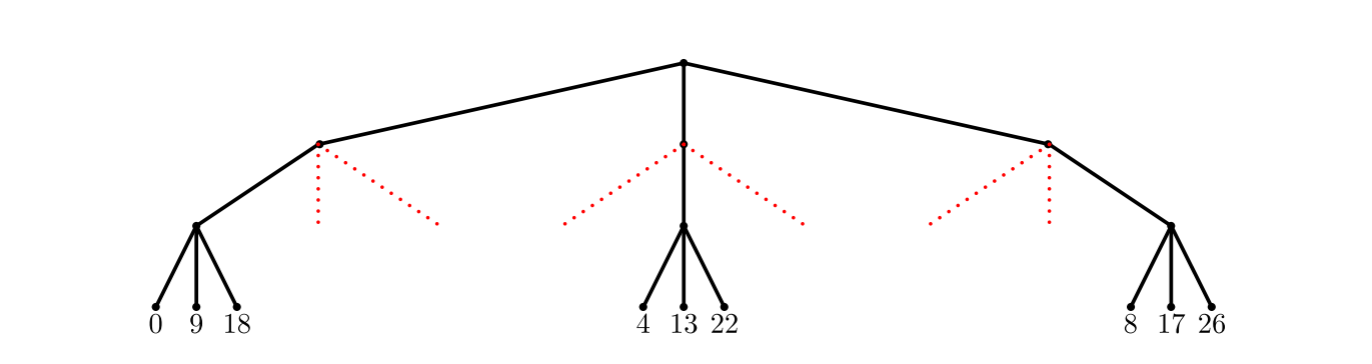}
  	\caption{Consider set $\{0, 4, 8, 9, 13, 17, 18, 22, 26\}$ as a $p$-homogeneous tree.}
  	\label{fig:3}
  \end{figure}

\begin{exam}
  Let \( p = 3 \), \( n = 3 \), \( C = \{0, 4, 8, 9, 13, 17, 18, 22, 26\} \) (see figure \ref{fig:3}). We have:
\begin{align*}
0 &= 0 \cdot 1 + 0 \cdot 3 + 0 \cdot 3^2, \quad
4 = 1 \cdot 1 + 1 \cdot 3 + 0 \cdot 3^2, \quad
8 = 2 \cdot 1 + 2 \cdot 3 + 0 \cdot 3^2, \\
9 &= 0 \cdot 1 + 0 \cdot 3 + 1 \cdot 3^2, \quad
13 = 1 \cdot 1 + 1 \cdot 3 + 1 \cdot 3^2, \quad
17 = 2 \cdot 1 + 2 \cdot 3 + 1 \cdot 3^2, \\
18 &= 0 \cdot 1 + 0 \cdot 3 + 2 \cdot 3^2, \quad
22 = 1 \cdot 1 + 1 \cdot 3 + 2 \cdot 3^2, \quad
26 = 2 \cdot 1 + 2 \cdot 3 + 2 \cdot 3^2.
\end{align*}
\end{exam}

A criterion for a subset $C \subset \Z_{p^n}$ to be $p$-homogeneous is given in \cite{FFS16}.
\begin{lem}[{\cite[Theorem 2.9]{FFS16}}] \label{thm5.4}
Let $n$ be a positive integer, and let $C \subset \Z_{p^n}$ be a multiset. Suppose that
\begin{enumerate}[{\rm(1)}]
    \item $|C| \le p^k$ for some integer $k$ with $1 \le k \le n$;
    \item there exist $k$ integers $1 \le j_1 < j_2 < \dots < j_k \le n$ such that
        \[
        \sum_{c \in C} e^{2\pi i c p^{-j_t}} = 0 \quad \text{for all } 1 \le t \le k.
        \]
\end{enumerate}
Then $|C| = p^k$ and $C$ is $p$-homogeneous. Moreover, the tree ${\mathcal T}_C$ is a ${\mathcal T}_{I}$-form tree with branched level set  $I = \{j_1-1, j_2-1, \dots, j_k-1\}$.
\end{lem}

\begin{lem}[{\cite[Theorem 4.2]{FFS16}}]\label{homogeneous}
Let $n$ be a positive integer, and let $C \subset \Z_{p^n}$. Then $C$ tiles  $\Z_{p^n}$ if and only if $C$ is p-homogeneous.
\end{lem}

\section{The structure of tiles on $\Z_{p^nq}$}
In this section, we primarily utilize the results established in \cite{FFS16} and \cite{CM99} to characterize the structure of tiles on
$\Z_{p^{n}q}$.

The structural properties of tiles in $\Z_{p^n}$ are characterized by $p$-homogeneity, as showed in \cite{FFS16}. Furthermore, the equivalence of
spectral sets and tiles in $\Z_{p^{n}q}$, where $p$ and $q$ are distinct primes, was proven in \cite{CM99}. This fundamental fact serves as a cornerstone in
subsequent proofs within this section.

\subsection{CM condition for tiles in finite group $\Z_N$.}
Let $A \subseteq \Z_N$ be a multi-set, and let $m_a$ denote the multiplicity of $a \in A$. The \emph{mask polynomial} of $A$ is defined as
\[
A(x) = \sum_{a \in A} m_a x^a.
\]
Denote by $\omega_N=e^{2\pi i /N}$,  which is a primitive $N$-th root of unity.
For any $d\in \{0, 1, \dots, N-1\}$, it is straightforward to verify that $A(\omega_N^d) = 0$ is equivalent to $\widehat{1_A}(d) = 0$.

Denote by $\Phi_s(x)$  the $s$-th cyclotomic polynomial.
\begin{lem}[{\cite[Lemma 2.4]{S19}}]\label{lem5.1}
Let $A$ be a subset of $\Z_N$. Let $p$ be a prime factor of $N$, and $a \in \Z_N$. Then the following statements hold.
\begin{enumerate}[{\rm (1)}]
    \item $A(\omega_N^a) = 0$ if and only if $A(\omega_N^{ag}) = 0$ for any $g \in \Z_N^*$.
    \item For any $d \mid N$,  $A(\omega_N^d) = 0$ if and only if  $\Phi_\frac{N}{d}(X) \mid A(X).$
    \item Suppose that $|\{p^d \in \Z_N: A(\omega_{N}^{N/p^d})= 0\}| = k$. Then $p^k \mid |A|$.
\end{enumerate}
\end{lem}

Let $S$ denote the set of prime powers dividing $N$, and define
\[
S_A = \{s \in S : \Phi_s(x) \mid A(x)\}.
\]
Coven and Meyerowitz introduced the following two properties in \cite{CM99}, which play a crucial role in characterizing  a set $A$ that tile $\Z_N$ by
translations:

\begin{itemize}
\item[(T1)] $|A| = A(1) = \prod_{s \in S_A} \Phi_s(1)$.
\item[(T2)] Let   $s_1, \dots, s_m \in S_A$ be powers of different primes. Then the polynomial $\Phi_{s_1 \dotsm s_m}(x)$ divides $A(x)$.
\end{itemize}
The following results are established in \cite{CM99}.
\begin{theo}[\cite{CM99}]Let  $\Omega\subset \Z_N$.
\begin{itemize}
\item If $\Omega$ satisfies properties \textbf{\rm (T1)} and \textbf{\rm (T2)}, then $\Omega$ tiles $\Z_N$ by translations.
\item If $\Omega$ tiles $\Z_N$ by translations, then \textbf{\rm(T1)} holds.
\item If $\Omega$ tiles $\Z_N$ by translations and $|\Omega|$ has at most two prime factors, then \textbf{\rm(T2)} holds.
\end{itemize}
\end{theo}
\begin{coro}\label{cor:T1T2}
For $N=p^n q$,  any tile $\Omega$ in $\Z_N$ satisfies properties \textbf{\rm (T1)} and \textbf{\rm (T2)}.
\end{coro}

\subsection{The structure of tiles in $\Z_{p^nq}$}

Now, let $N=p^nq$ with $p$, $q$ are distinct primes. For a set $A\subset\Z_{p^{n}q}$,
denote ${\mathcal Z}_A=\{x\in \Z_{p^{n}q}:  {\widehat{1_A}}(x) = 0 \} $.
Let \[\mathcal{I}_A=\big\{0\leq a \leq n-1:\ p^aq \in {\mathcal Z}_A \big\}.\]

\begin{lem}
Let  $(\Omega,T)$ be a tiling pair in $\Z_{p^{n}q}$.  Then  the cardinality of $\Omega$ is either $p^t$ or $p^tq$,
where $0 \leq  t \leq  n$. In both cases $|\mathcal{I}_\Omega|=t$ and $\mathcal{I}_\Omega$ and $\mathcal{I}_T$ forms
a disjoint union of $\{0, 1, \cdots, n-1\}$.
\end{lem}
\begin{proof}
Since  $(\Omega,T)$ is a tiling pair  in $\Z_{p^{n}q}$, it follows that $|\Omega| \mid  p^nq$.
Hence $|\Omega|=p^t$ or $|\Omega|=p^t q$ for some  $0 \leq  t \leq  n$.
Note that for any $0\leq a \leq n-1$, $p^aq \in \mathcal{Z}_\Omega\cup \mathcal{Z}_T$. By statement (3) of Lemma \ref{lem5.1},  $|\mathcal{I}_\Omega|=t$, $|\mathcal{I}_T|=n-t$ and
$\mathcal{I}_\Omega$ and $\mathcal{I}_T$ forms  a disjoint union of $\{0, 1, \cdots, n-1\}$.
\end{proof}

 Note that $\Z_{p^{n}q}$ is isomorphic to the group $\Z_{p^n} \times \Z_q$ by the isomorphism
 \[ x \mapsto (x_1, x_2) \quad \hbox{ with }  x_1\equiv x \! \!\!\pmod{p^n}, x_2\equiv x  \! \!\! \pmod q. \]  Let $\pi_1: \Z_{p^{n}} \times \Z_{q} \to \Z_{p^n} $  and
 $\pi_2: \Z_{p^{n}} \times \Z_{q} \to \Z_{q}  $ be the projection maps.
 In the reminder of this section, we consider the  product  form $\Z_{p^n} \times \Z_q$.

 For a function $f$ on  $\Z_{p^n} \times \Z_q$, the Fourier transform of $f$  is
 \[\widehat{f}(x,y)=\sum_{(u,v)\in\Z_{p^n} \times \Z_q} f(x,y)e^{2\pi i( \frac{ux}{p^n}+\frac{vy}{q})}.\]
   Hence, it follows that   $(p^a, 0)\in \mathcal{Z}_A$ if and only if $a\in  \mathcal{I}_{A}$.

% \begin{prop}\label{prop5.7}
%Let  $(\Omega,T)$ be a tiling pair in   $\Z_{p^n} \times \Z_q$.   For $0\leq s \leq q-1$, let $\Omega_j=\{x\in \mathbb{Z}/p^n\Z : (x,j)\in \Omega\}$.
%\begin{enumerate}[{\rm(1)}]
%    \item If   $\#\Omega=p^t$, then  $\Omega_j$ are disjoint and $\cup_{j=0}^{q-1}\Omega_j$ tiles $\Z_{p^n}$ by translation.
%    \item If $\#\Omega=p^tq$, then  all  $\Omega_j$  can tiles  $\Z_{p^n}$ with a common translation set.
%    \end{enumerate}
%\end{prop}

\begin{proof}[Proof of Theorem \ref{main1} ]
(1) Case $|\Omega|=p^t$. Note that
 \[\widehat{1_{\Omega}}(p^a, 0)=\sum_{(u, v)\in\Z_{p^n} \times \Z_q} e^{2\pi i \frac{u}{p^{n-a}}}.\]
 Consider $\pi_1(\Omega)$ as a multi-set in $\Z_{p^n}$, where the multiplicity of $x \in \Z_{p^n}$ is given by $|\pi_1^{-1}(x)|$. Therefore,
 $\widehat{1_{\Omega}}(p^a,0) = 0$ if and only if $\widehat{1_{\pi_1(\Omega)}}(p^a) = 0$.

% Consider $\pi_1(A)$ as a multi-set in $\mathbb{Z}/p^n\Z$,  where the multiplicity  of $x \in \mathbb{Z}/p^n\Z$ is given by  $\#\pi_1^{-1}(x)$.
% Hence $\widehat{\1_{A}}(p^a,0)=0$ if and only if $\widehat{\1_{\pi_1(A)}}(p^a)=0$.

By Lemma \ref{thm5.4},  it follows that the multiplicity of each point  $\pi_1(\Omega)$ is one and $\pi_1(\Omega)$ corresponds to a   ${\mathcal T}_I$-form $p$-homogeneous tree with $I = n-1-\mathcal{I}_\Omega$. %and $J = \{0, 1, \dots, n-1\} \setminus I$.

(2) Case $|\Omega|=p^tq$.  Assume that $(\Omega, T)$ is a tiling pair. Then we have $|T|=p^{n-t}$. By statement (3) of Lemma \ref{lem5.1},
$\Omega(\omega_{N}^{p^n})=0$, which is equivalent to $\Phi_{q}(X) \mid \Omega(X).$
Hence, for any $a\in I_\Omega$, by Corollary \ref{cor:T1T2} and  statement (2) of Lemma \ref{lem5.1}, we have  $\Omega(\omega_{N}^{p^a})=0$,
which is equivalent to $(p^a,j )\in \mathcal{Z}_\Omega$ with $j \equiv p^a \pmod{q}.$  Actually, this implies that $(p^a, j )\in \mathcal{Z}_\Omega$ for all $a \in \mathcal{I}_\Omega$ and $j\in \{0, \cdots, q-1\}$.
Moreover, it is clear that  $(0,j )\in \mathcal{Z}_\Omega$ for all $j\in \{1, \cdots, q-1\}$. Take a $p$-homogeneous ${\mathcal T}_{J}$-form subset  $T_0\subset  \Z_{p^n}$. Then $(\Omega, T_0\times\{0\})$ forms a tiling pair, which implies that each  $(\Omega_j, T_0)$ forms a tiling pair of $\Z_{p^n}$.
\end{proof}

\section{The structure of tiles in $\Z_{p^n}\times\Z_p$}

\subsection{Equidistribution property}
%Let $\xi_n=e^{2\pi i\frac{1}{p^n}}$ be a primitive $p^n$-th root of unity.
 For $\textbf{x}=(x_1,x_2),~\textbf{y}=(y_1,y_2)\in \Z_{p^n}\times \Z_{p}$,
we define the {\em inner product} in $\Z_{p^n}\times \Z_{p}$ by the formula
\[\langle \textbf{x},\textbf{y} \rangle= x_1y_1 + p^{n-1}x_2y_2\in \Z_{p^n}.\]

We define
\[H (\textbf{d},t):=\{\mathbf{x}\in \Z_{p^n}\times \Z_{p}: \langle \mathbf{x},\textbf{d}\rangle=t\},\]
for $\textbf{d}\in \Z_{p^n}\times \Z_{p}$ and $t\in \Z_{p^n}$. We call such set a \textbf{plane} in $\Z_{p^n}\times \Z_{p}$.

For $r\in\Z_{p^n}^\ast$ and $\textbf{d}=(d_1,d_2)\in \Z_{p^n}\times \Z_{p}$, we define the \textbf{scalar} product as:
\[r\textbf{d}=(\tilde{d}_1, \tilde{{d}}_2)\in \Z_{p^n}\times \Z_{p},\]
where $\tilde{{d}}_1\equiv r {d}_1\bmod p^n$ and $\tilde{ {d}}_2\equiv r {d}_2 \bmod p$.

The following lemma provide the equidistribution property of a set $A\subset \Z_{p^n}\times \Z_{p}$.

\begin{lem}[{\cite[Lemma 3.1]{S19}}]\label{52}
Let $A\subseteq \Z_{p^n}\times \Z_{p}$ and $\mathbf{d}\in \Z_{p^n}\times \Z_{p}$. The following are equivalent:
\begin{enumerate}[{\rm(1)}]
  \item $\widehat{1_A}(\mathbf{d})=0$;

  \item $\widehat{1_A}(r\mathbf{d})=0$, for any $r\in \Z_{p^n}^\ast$;

  \item $|A\cap H(\mathbf{d}, t)|=|A\cap H(\mathbf{d}, t')|$, if $t\equiv t' \bmod p^{n-1}$.

\end{enumerate}
\end{lem}

For $\mathbf{u},\mathbf{v} \in \Z_{p^n}\times \Z_{p}$, we  define the relation $\mathbf{u}\sim \mathbf{v}$, if there exists $r\in \Z_{p^n}^\ast$ such that
$\mathbf{u}=r\mathbf{v}$. Thus, the equivalent classes in $\Z_{p^n}\times \Z_{p}$ by $``\sim"$ are
\[(1, 0),(c, p^i) \quad \text{for all}~c\in \Z_{p^n}\times \Z_{p} \quad\text {and} \quad i\in \{0, 2, \cdots, n-1\}.\]

Thus, by Lemma \ref{52}, when we study the set of zeros $\mathcal{Z}_A$ of a set $A\subseteq \Z_{p^n}\times \Z_{p}$,
we only need to consider the elements which have the above forms.

Now we give the divisibility property for a set in $\Z_{p^n}\times \Z_{p}$.

\begin{lem}[\cite{Z2022} Lemma 3.2] \label{lem-3.2}
Let $A\subseteq \Z_{p^n}\times \Z_{p}$. If $(p^{i_1}, a), (p^{i_2}, 0), \cdots, ( p^{i_s}, 0)\in \mathcal{Z}_A$
for some $a \in \Z_{p^n}$ and $0 \leq {i_1} < {i_2} < ... < {i_s}\leq n-1$, then $p^s\mid  |A|$.
\end{lem}

\subsection{Proof of the theorem \ref{main2}}

In this subsection, we are concerned tiles in $\Z_{p^n}\times \Z_{p}$. Let $\Omega$ be a non-trivial tile in $\Z_{p^n}\times \Z_{p}$, then, due to Lemma \ref{2.6},
we have that $|\Omega|\big|p^{n+1}$. Thus, we can assume that $|\Omega|=p^t$ for some  $1 \leq t\leq n$.

 For a set $A\subset \Z_{p^n}\times \Z_p$, denote
 $${\mathcal Z}_A= \big\{(x,y)\in \Z_{p^n}\times \Z_p:  \widehat{1_A}(x,y) = 0 \big\}.$$

Let $$ \mathcal{I}_A=\big\{0\leq i \leq n-1: (p^i,0)\in {\mathcal Z}_A \big\}.$$

\begin{lem}
Let  $\Omega$ be a tile of  $\Z_{p^{n}}\times \Z_p$ with  $|\Omega|=p^t$  for some  $1 \leq t\leq n$. Then  $ |\mathcal{I}_\Omega|=t-1 $ or $t$.
\end{lem}
\begin{proof}
Let $T$ be a tiling complement of $\Omega$ in $\Z_{p^n}\times \Z_{p}$.  Then $|T|=p^{n-t+1}$.
%Let
%$$
%I=\{i\in [0, n-1]: (p^i, 0)\in \mathcal{Z}_\Omega \}, \ J=\{i\in [0, n-1]: ( p^i, 0)\in \mathcal{Z}_T \}.
%$$
By Lemma \ref{lem-3.2}, we have that $|\mathcal{I}_\Omega|\leqslant t $ and $|\mathcal{I}_T|\leqslant n-t+1 $.
On the other hand, since
$$
\mathcal{Z}_\Omega\cup \mathcal{Z}_T= \Z_{p^n}\times \Z_{p} \backslash  \{ (0, 0)\},
$$
we have $|\mathcal{I}_\Omega|+|\mathcal{I}_T| \geqslant n$. Thus, $t-1 \leqslant |\mathcal{I}_\Omega| \leqslant t $, that means $|\mathcal{I}_\Omega|=t$ or $t-1$.
\end{proof}

Define a map $\pi_1$ from $\Z_{p^n}\times \Z_{p}$ to $\Z_{p^n}$ by
\[\pi_1(a, b)=a,  \hbox{ for }  (a, b)\in  \Z_{p^n}\times \Z_{p}.\]
\begin{prop}\label{case t-1}
Let $\Omega$ be a tile in $\Z_{p^n}\times \Z_{p}$ with $|\Omega|=p^t$  and $|\mathcal{I}_\Omega|=t$. Then $\pi_1(\Omega)$ is  a $p$-homogeneous set  in $\Z_{p^n}$  with  $| \pi_1(\Omega)|=p^t $.
\end{prop}
\begin{proof}
Write
$ \Omega=\bigsqcup\limits_{j=0}^{p-1}(  \Omega_{j}  \times \left \{ j \right \})$ with $\Omega_{j} \subset \Z_{p^n}$.
For each $i\in \mathcal{I}_{\Omega}$, we have

\begin{align*}
\widehat{1_\Omega}(p^i, 0)
	&=   \sum_{j=0}^{p-1} \widehat{1_{\Omega_j \times \{j\} }}(p^i,0)\\
        &=  \sum_{j=0}^{p-1} \widehat{1_{\Omega_j }}(p^i)\\
	&=  \widehat{1_{\Omega_0 \cup \Omega_1 \cup  \cdots  \cup \Omega_{p - 1}}}(p^i)\\
	&= 0.
\end{align*}
By Lemmas \ref{thm5.4} and \ref{homogeneous}, $\pi_1(\Omega)= {\Omega_0} \cup {\Omega_1} \cup  \cdots  \cup {\Omega_{p-1}}$
is a $p$-homogeneous set in $\Z_{p^n}$ and $| \pi_1(\Omega)|=p^t$.
\end{proof}

Now assume $|\mathcal{I}_{\Omega}|=t-1$. Let  $\mathcal{J}=\{0,1,\cdots n-1\}\backslash \mathcal{I}_{\Omega}$.

\begin{lem} \label{lem-min}
Let $\Omega$ be a tile in $\Z_{p^n}\times \Z_{p}$ with $|\Omega|=p^t$  and  $|\mathcal{I}_{\Omega}|=t-1$ for some  $1\leq t\leq n$.
If for each $i\in \{0,1, \cdots, n-1\}\setminus \mathcal{I}_{\Omega}$ and $b \in \Z_p^*$,  $(p^{i}, b)\notin  \mathcal{Z}_{\Omega}$, then $(0, 1) \in \mathcal{Z}_{\Omega}$.
%Then, when $|I|=t-1 $ and $C_I\subseteq \mathcal{Z}_A $, $C_J \cap \mathcal{Z}_A = \emptyset$, $A_i$ is a tile in $\Z_{p^n}$ and $|A_i|=p^{t-1}$.
 %Where ${A_i} = \{x: (x, i) \in A\}$, $i=0, \cdots,\ p-1$.
\end{lem}
\begin{proof}
Let $T$ be a tiling complement of $\Omega$.   By assumption, we have
\[ \{(p^j, b): j\in \mathcal{J}, b \in \Z_p \} \subset  \mathcal{Z}_{T}.\]
If $(0, 1) \in \mathcal{Z}_T $, we define
 $$
 \Lambda = \big\{s_n (0, 1) + (\sum_{j\in \mathcal{J}} s_j p^{j}, 0): s_j, s_n \in \{0, 1, \cdots,  p-1\}\big\}.
 $$
Then, for any $ \lambda \neq \lambda' \in \Lambda$, we have
$$
\lambda - \lambda' = r_n (0, 1) + (\sum_{j\in \mathcal{J}} r_j p^{j},\ 0),
$$ where $r_j, r_n  \in \{-p + 1, \cdots,\ p - 1\}$.
Notice that
\[ r_n( 0,1) + (\sum_{j\in J} r_j p^{j},\ 0) \sim \begin{cases} (0, 1), & \text{if} \ r_n \neq 0 \ \text{and} \ r_j = 0\ \text{for}\ j \in \mathcal{J}; \\
 (p^l,r_{l}^{-1}r_n), & \text{if} \ r_n \neq 0, r_l \neq 0  \ \text{and} \ r_j = 0\ \text{for} \ 1 \leq j < l; \\
 (p^{l},\ 0), & \text{if} \ r_n = 0,\ r_l \neq 0 \  \text{and} \ r_j = 0\ \text{for} \ 1 \leq j < l. \end{cases} \]
Therefore,  \((\Lambda - \Lambda)\setminus\{(0, 0)\} \subseteq \mathcal{Z}_T \), which implies  the characters  $\{\chi_{\lambda}\}_{\lambda \in \Lambda}$  are orthogonal in $L^{2}(T)$.  However, $|\Lambda |= p^{n-t+2}>|T|$,
which is a contradiction, since the dimension of  $L^2(T)$ is $|T|$. Thus, \(( 0, 1) \in \mathcal{Z}_\Omega\).
\end{proof}

Note that $p^n=0$ in $\Z_{p^n}$. Let $\Omega$ be tile of $\Z_{p^n}\times \Z_p$  with $|\Omega|=p^t$ and $|\mathcal{I}_{\Omega}|=t-1$.
Define
\[\gamma_{\Omega}= \min \big\{j\in  \{0,1,\cdots, n\}\setminus \mathcal{I}_{\Omega}:  (p^j, b)\in \mathcal{Z}_\Omega \hbox{ for some }b \in \Z_{p}^{*}\big\}.\]
By Lemma  \ref{lem-min},  $\gamma_{\Omega}$ is well defined.

\begin{lem}\label{lem-min2}
Let $\Omega$ be a tile in $\Z_{p^n}\times \Z_{p}$ with $|\Omega|=p^t$  and  $|\mathcal{I}_{\Omega}|=t-1$ for some  $1\leq t\leq n$.
Then for any tiling complement $T$, we have  $(p^i, b )\notin \mathcal{Z}_T $
for all $i\in \mathcal{I}_{\Omega}$ with $i<\gamma_{\Omega}$, $b \in \{ 0,1, \cdots, p-1 \}.$ Consequently,  $(p^i, b )\in \mathcal{Z}_{\Omega}$.
\end{lem}
\begin{proof}
 Let $T$ be a tiling complement of $\Omega$ and  $\mathcal{J}= \{0,1, \cdots, n-1\}\setminus \mathcal{I}_{\Omega}$.
 Assume  that  there is an $i\in \mathcal{I}_{\Omega}$ with $i< \gamma_{\Omega}$ such that $(p^i, b)\in \mathcal{Z}_T$ for $b\in \Z_p$.
  Define
$$
\Lambda=\Big\{(\sum \limits_{j\in \mathcal{J}}s_jp^j, 0)+s_i (p^i,b): s_i, s_j \in \{0, 1, \cdots, p-1\} \Big\}.
$$
For any $\lambda\neq\lambda' \in \Lambda $, we have
$$\lambda-\lambda' = (\sum \limits_{j\in \mathcal{J}} r_j p^j,0)+r_i (p^i,b),
$$
where $r_i, r_j \in \{-p+1, \cdots, p-1\}$. Observe that
\[\sum_{j\in  \mathcal{J}} r_j (p^{j}, 0) +~r_i( p^{i},b) \sim ( p^{i},b) \]
if $\ r_i \neq 0 $ and  $r_j = 0$ for $j < i$. And
 \[\sum_{j\in  \mathcal{J}} r_j (p^{j}, 0) +~r_i( p^{i},b) \sim ( p^{l} ,r_{l}^{-1} r_i b) \]
  if $l$ is the minimal number such that $r_l\neq 0$ and $l<i$.
Therefore \[(\Lambda-\Lambda )\backslash \{(0, 0)\} \subseteq \mathcal{Z}_T,\]
which implies  the characters  $\{\chi_{\lambda}\}_{\lambda \in \Lambda}$  are orthogonal in $L^{2}(T)$.
 However, $|\Lambda |= p^{n-t+2}>|T|$,
which is a contradiction, since the dimension of  $L^{2}(T)$ is $|T|$. %Thus, \((p^i, b) \in \mathcal{Z}_\Omega\)
%\begin{align*}
%& \sum_{j\in  \mathcal{J}} r_j (p^{j}, 0) +~r_i( p^{i},b) \\
%\sim &
%\begin{cases}
%( p^{i},b), &\text{if} \ r_i \neq 0 \ \text{and} \ r_j = 0\ \text{for}\ j < i; \\
%( p^{l} ,r_{l}^{-1} r_i b), &\text{if} l is the minimal number   r_{l}\neq 0  \text{ for some } 0<l<i  \\\text{and}\ r_j = 0\ \text{for} \ 1 \leq j < l; \\
%%( p^{l},0), &\text{if} \ r_i = 0, \ r_l \neq 0 \ \ \text{and}\ r_j = 0\ \text{for} \ 1 \leq j < l.
%\end{cases}
%\end{align*}
%So $(\Lambda-\Lambda )\backslash \{(0, 0)\} \subseteq \mathcal{Z}_T$. Then it follows that $\Lambda $ is a spectrum of $T$, however $|\Lambda |= p^{n-t+2} > |T|$,
%which is a contradiction.
\end{proof}

We distinguish two cases:
\[ (1) \ \gamma_{\Omega}=n,  \quad (2) \  \gamma_{\Omega}<n.\]

Firstly, we deal with the case $\gamma_{\Omega}=n$.
\begin{prop}\label{prop4.6}
Let $\Omega$ be tile of $\Z_{p^n}\times \Z_p$  with $|\Omega|=p^t$, $|\mathcal{I}_{\Omega}|=t-1$ and $\gamma_{\Omega}=n.$
Then  for each $b\in \Z_p$, the set $\Omega_b=\{x\in \Z_{p^n}: (x, b)\in \Omega \}$ is a  $p$-homogeneous subset of $\Z_{p^n}$ with branched level set $n-1-\mathcal{I}_{\Omega}$.
\end{prop}
\begin{proof}
By Lemmas  \ref{lem-min}  and  \ref{lem-min2}, we have
\[  \left\{(0, b):b \in \Z_p \right\}\subset \mathcal{Z}_{\Omega}\]
and
\[ \left\{(p^i,b): i \in \mathcal{I}_{\Omega}, b \in \Z_p \right\}\subset \mathcal{Z}_{\Omega}.\]

Let $T_0$ be a subset of $\Z_{p^n}$ with  $p^j\in \mathcal{Z}_{T_0}$ for $j \in \mathcal{J}= \{0,\cdots, n-1\} \setminus  \mathcal{I}_{\Omega}$.
Take $T=T_0\times\{0\}$, which is a subset of $\Z_{p^n}\times \Z_p$. By calculating the zero set $\mathcal{Z}_{T}$, it follows that $(\Omega,T)$ is a tiling pair of $\Z_{p^n}\times \Z_p$.  Hence, for each $b\in \Z_p$, $(\Omega_b, T_0)$ is a tiling pair of $\Z_{p^n}$, which implies
$\Omega_b=\{x\in \Z_{p^n}: (x, b)\in \Omega \}$ is a  $p$-homogeneous subset of $\Z_{p^n}$ with branched level set $n-1-\mathcal{I}_{\Omega}$.
\end{proof}

Now, we shall deal the case $\gamma_{\Omega}<n$.

\begin{prop}\label{prop4.7}
Let $\Omega$ be a tile in $\Z_{p^n}\times \Z_p$ with $\gamma_{\Omega}<n$, $|\Omega|=p^t$  and
 $|\mathcal{I}_{\Omega}|=t-1$ for some $1 \leq t \leq n-1$. Assume that $(p^{\gamma_{\Omega}}, \alpha)\in \mathcal{Z}_{\Omega}$ and  let $\mathcal{I}=\mathcal{I}_{\Omega}\cup \{\gamma_{\Omega}\}$.
   Then
the set \[\widetilde{\Omega} = \left\{ x +\alpha y p^{n-\gamma_{\Omega}-1}: (x, y) \in \Omega\right\}\] is a $p$-homogeneous set in $\Z_{p^n}$.
\end{prop}
\begin{proof}

%Let $\mathcal{I}_{\Omega}=\{i_1,\cdots,i_{t-1}\}$ and $i_1<\cdots<i_s<\gamma_{\Omega}<i_{s+1}<\cdots<i_{t-1}$.
%By Lemma \ref{lem-min2}, we have $(p^{i_k},\alpha) \in \mathcal{Z}_\Omega$ for all $1\leq k \leq s$ and $\alpha\in\{0, \cdots, p-1\}$.

Consider the multi-set \[\widetilde{\Omega} = \left\{ x +\alpha y p^{n-\gamma_{\Omega}-1}: (x, y) \in \Omega\right\} \subset  \Z_{p^n}.\]
It is obvious that $|\widetilde{\Omega}|=p^t$.  Now, we  shall show that   $p^{j}\in \mathcal{Z}_{\widetilde{\Omega} }$ for $j \in \mathcal{I}$.

For $j \in \mathcal{I}_{\Omega}$ with $j> \gamma_{\Omega}$,   we have
\[ \widehat{1_{\widetilde{\Omega}}}(p^{j})=\sum_{(x,y)\in \Omega}e^{2 \pi i\frac{ (x +\alpha y p^{n-\gamma_{\Omega}-1})p^{j}}{p^n}}=\sum_{(x,y)\in \Omega}e^{2 \pi i\frac{ x p^{j}}{p^n}}=\widehat{1_\Omega}(p^{j}, 0)=0.\]

  For $j=\gamma_{\Omega}$, we have
  \[ \widehat{1_{\widetilde{\Omega}}}(p^{\gamma_{\Omega}})=\sum_{(x,y)\in \Omega}e^{2 \pi i\frac{ (x +\alpha y p^{n-\gamma_{\Omega}-1})p^{\gamma_{\Omega}}}{p^n}}=\sum_{(x,y)\in \Omega}e^{2 \pi i\frac{ x p^{\gamma_{\Omega}}+\alpha y p^{n-1}}{p^n}}=\widehat{1_\Omega}(p^{\gamma_{\Omega}}, \alpha)=0.\]

  Now we consider the case $j<\gamma_{\Omega}$.  By the definition of  $\gamma_{\Omega}$,  we have  $(p^j, b)\in \mathcal{Z}_{\Omega}$ for each $b \in \Z_p$. For each $b\in \Z_p$, let $\Omega_b=\{x\in \Z_{p^n}: (x,b)\in \Z_{p^n}\times \Z_p\}$.  By Lemma \ref{lem:periodic}, we have $p^j \in \mathcal{Z}_{\Omega_b}$ for each $b\in \Z_p$. By the definition of $\widetilde\Omega$, we have
  \begin{align*}
    \widehat{1_{\widetilde{\Omega}}}(p^{j})&=\sum_{(x,y)\in \Omega}e^{2 \pi i\frac{ (x +\alpha y p^{n-\gamma_{\Omega}-1})p^{j}}{p^n}} \\
    &=\sum_{b\in \Z_p}\sum_{x\in \Omega_b}e^{2 \pi i\frac{ x p^{j}+\alpha b p^{n+j-1-\gamma_{\Omega}}}{p^n}} \\
    &= \sum_{b\in \Z_p}\widehat{1_{\Omega_b+\alpha b p^{n+j-1-\gamma_{\Omega}}}}(p^j).
   \end{align*}
By Lemma \ref{lem-traninv}, we have $\widehat{1_{\widetilde{\Omega}}}(p^{j})=0$. Hence,  $\{p^j: j \in \mathcal{I}\} \subset \mathcal{Z}_{\widetilde{\Omega}}$.
Note that $|\mathcal{I}|=t$ and $|\widetilde{\Omega}|=p^t$.  By Lemma  \ref{thm5.4},  $\widetilde{\Omega}$  is $p$-homogeneous with branch level set  $n-1-\mathcal{I}.$
\end{proof}
\begin{lem}\label{lem4.9}
Let $\Omega$ be a tile in $\Z_{p^n}\times \Z_p$. Assume that $\mathcal{I}_{\Omega}=\{i_1,i_2, \cdots, i_{t-1}\}$ with  $i_1< \cdots<i_s<\gamma_{\Omega}<i_{s+1}<\cdots <i_{t-1}$. Then for each $(x,y)\in \Omega$, the sets \[\Omega_{x,y}:=\{x^{\prime }\in\Z_{p^n}: (x^{\prime},y)\in \Omega  \text{ and } x^{\prime}\equiv x \!\!\mod p^{n-\gamma_{\Omega}-1}  \
\}\]  is  $p$-homogenous with  the branched level set $n-1-\{i_1,i_2 \cdots,i_s\}$.
\end{lem}
\begin{proof}
Assume that $\mathcal{I}_{\Omega}=\{i_1,i_2, \cdots, i_{t-1}\}$ with  $i_1< \cdots<i_s<\gamma_{\Omega}<i_{s+1}<\cdots <i_{t-1}$.

We shall characterize  the structure of  the tile $\Omega$ by induction. We distinguish two cases.

\begin{figure}[h]
  	\centering
  	\includegraphics[width=1.1\linewidth]{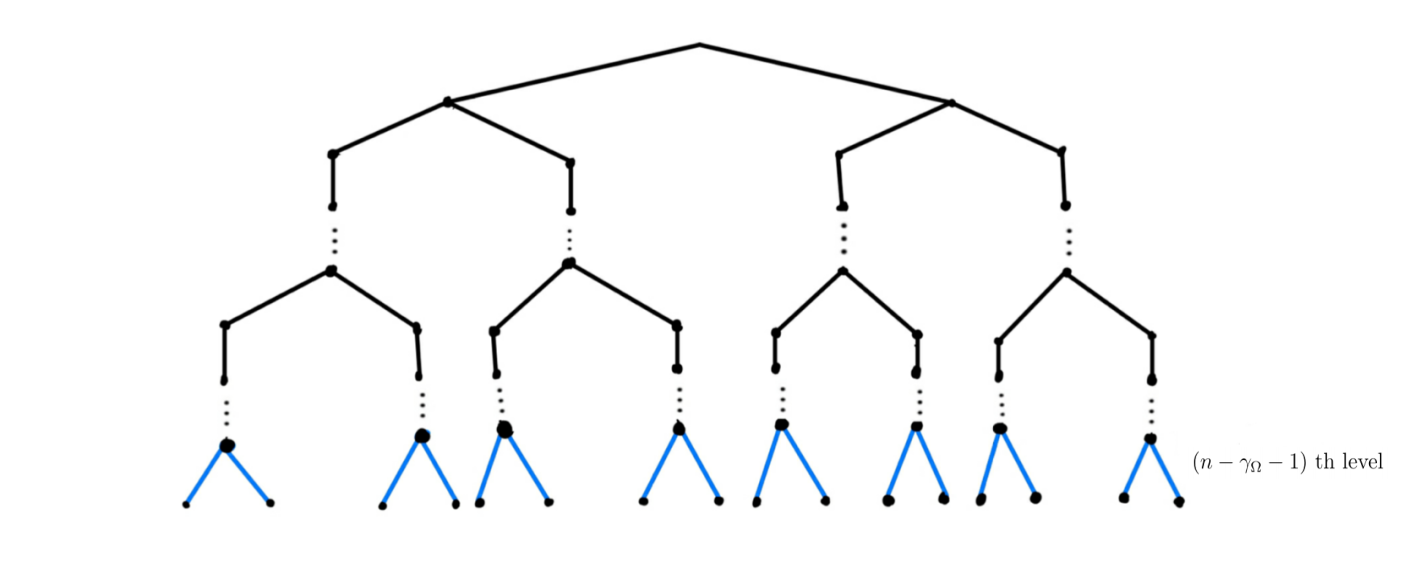}
  	\caption{The corresponding tree  of $\widetilde{\Omega}$, where the blue edges determined by the hyperplanes ${H}((p^{\gamma_{\Omega}},\alpha), k)$.}
  	\label{fig:4}
  \end{figure}

%\textbf{Case 1:} Assume $i_1<i_2<\cdots<i_{t-1}<\gamma_{\Omega}$.
%By Lemma \ref{52}, we get
%\begin{equation}\label{hp0}
% |\Omega\cap {H}((p^{\gamma_{\Omega}},\alpha), k)|=|\Omega \cap {H}((p^{\gamma_\Omega}, \alpha), k+jp^{n-1})|.
%\end{equation}
%
%The set $\widetilde{\Omega}$ corresponding  to a $p$-homogeneous tree with branched level set \[n-1-\{\gamma_{\Omega},i_1,i_2,\cdots,i_{t-1}\}.\]
%The branched levels   $(n-i_1-1), (n-i_2-1),\cdots (n-i_{t-1}-1)$ is determined by the projection $\pi_1(\Omega)$, that is,
%the vertices at  $(n-i_1-1)$th, $(n-i_2-1)$$th, \cdots$, $(n-i_{t-1}-1)$th levels of the corresponding tree of the set $\pi_1(\Omega)$ have $p$ descendants.
%On the other hand, there is $\alpha \in \{ 1, \cdots, p-1 \}~\text{such that }~(p^{\gamma_{\Omega}},\alpha)\in \mathcal{Z}_\Omega $. Thus, we have
%$$
%\widehat {1_{\Omega}}(p^{\gamma_{\Omega}}, \alpha)=\sum\limits_{(x, y) \in \Omega} {{e^{2\pi i \frac{xp^{\gamma_{\Omega}}+\alpha y p^{n-1} }{{p^{n}}}}}}=0.
%$$
%By Lemma \ref{52}, we get
%\begin{equation}\label{hp0}
% |\Omega\cap {H}((p^{\gamma_{\Omega}},\alpha), k)|=|\Omega \cap {H}((p^{\gamma_\Omega}, \alpha), k+jp^{n-1})|.
%\end{equation}
%By the above argument, we conclude that for $(x,y)\in \Omega$, the sets $\Omega_{x,y}=\{x\}$ is a single point.

\textbf{Case 1:} Assume $\gamma_{\Omega}<i_1<i_2<\cdots<i_{t-1}$.
The set $\widetilde{\Omega}$ corresponding  to a $p$-homogeneous tree with branched level set \[n-1-\{\gamma_{\Omega},i_1,i_2,\cdots,i_{t-1}\}.\]
The branched levels   $(n-i_1-1), (n-i_2-1),\cdots (n-i_{t-1}-1)$ is determined by the projection $\pi_1(\Omega)$, that is,
the vertices at  $(n-i_1-1)$th, $(n-i_2-1)$$th, \cdots$, $(n-i_{t-1}-1)$th levels of the corresponding tree of the set $\pi_1(\Omega)$ have $p$ descendants.
On the other hand, there is $\alpha \in \{ 1, \cdots, p-1 \}~\text{such that }~(p^{\gamma_{\Omega}},\alpha)\in \mathcal{Z}_\Omega $. Thus, we have
$$
\widehat {1_{\Omega}}(p^{\gamma_{\Omega}}, \alpha)=\sum\limits_{(x, y) \in \Omega} {{e^{2\pi i \frac{xp^{\gamma_{\Omega}}+\alpha y p^{n-1} }{{p^{n}}}}}}=0.
$$
By Lemma \ref{52}, we get
\begin{equation}\label{hp0}
 |\Omega\cap {H}((p^{\gamma_{\Omega}},\alpha), k)|=|\Omega \cap {H}((p^{\gamma_\Omega}, \alpha), k+jp^{n-1})|.
\end{equation}
By the above argument, we conclude that for $(x,y)\in \Omega$, the sets $\Omega_{x,y}=\{x\}$ is a single point.

\textbf{Case 2:} Assume $i_1<\cdots<i_s< \gamma_{\Omega}< i_{s+1} \cdots<i_{t-1}$.
By Lemma \ref{lem:periodic}, the condition $(p^{i_1}, b)\in \mathcal{Z}_{\Omega}$ for  each $b\in \Z_p$  implies that  $p^{i_1} \in \mathcal{Z}_{\Omega_b}$ for each $b$.
Then  for each $b\in \Z_p$,  either $\Omega_b =\emptyset$, or $p \mid |\Omega_b|$ and   $\Omega_b$ can decomposed into $|\Omega_b|/p$ subsets
satisfy that  each subset contains $p$ elements and each two distinct point $x, y$ in a  same subset such that  $p^{n-1-i_1}\mid (x-y)$ and  $ p^{n-i_1}\nmid (x-y)$.
By induction,  for each $b\in \Z_p$,  either $\Omega_b =\emptyset$, or $p^s \mid |\Omega_b|$ and   $\Omega_b$ can decomposed into $|\Omega_b|/p^s$ subsets  $\Omega^{s}_{b,\ell}, 0\leq  \ell  \leq |\Omega_b|/p^s$,
satisfy that  each subset corresponding to  a $p$-homogenous tree with branch level set $n-1-\{i_1, \cdots, i_s\}$.

 The set $\widetilde{\Omega}$ corresponding  to a $p$-homogeneous tree with branched level set \[n-1-\{i_1,\cdots,i_s, \cdots, \gamma_{\Omega},i_{s+1},\cdots, i_{t-1}\}.\]
The branched levels   $(n-i_{t-1}-1), (n-i_{t-2}-1),\cdots (n-i_{s+1}-1)$ is determined by the projection $\pi_1(\Omega)$, that is,
the vertices at  $(n-i_1-1)$th, $(n-i_2-1)$$th, \cdots$, $(n-i_{t-1}-1)$th levels of the corresponding tree of the set $\pi_1(\Omega)$ have $p$ descendants.

On the other hand, there is $\alpha \in \{ 1, \cdots, p-1 \}~\text{such that }~(p^{\gamma_{\Omega}},\alpha)\in \mathcal{Z}_\Omega $. Thus, we have
$$
\widehat {1_{\Omega}}(p^{\gamma_{\Omega}}, \alpha)=\sum\limits_{(x, y) \in \Omega} {{e^{2\pi i \frac{xp^{\gamma_{\Omega}}+\alpha y p^{n-1} }{{p^{n}}}}}}=0.
$$
By Lemma \ref{52}, we get
\begin{equation}\label{hp0}
 |\Omega\cap {H}((p^{\gamma_{\Omega}},\alpha), k)|=|\Omega \cap {H}((p^{\gamma_\Omega}, \alpha), k+jp^{n-1})|.
\end{equation}

Note that  for each  $\Omega^{s}_{b,\ell}, 0\leq  \ell  \leq |\Omega_b|/p^s$,  the set  $\Omega^{s}_{b,\ell}\subset {H}((p^{\gamma_{\Omega}},\alpha), k)$ for some $k$.
By counting the cardinality  of ${H}((p^{\gamma_{\Omega}},\alpha), k)$. We know that
$\Omega\cap {H}((p^{\gamma_{\Omega}},\alpha), k)$ is either empty or  $\Omega^{s}_{b,\ell}$ for some $b$ and $\ell$.
Hence, for each $(x,y)\in \Omega$, the sets \[\Omega_{x,y}:=\{x^{\prime }\in\Z_{p^n}: (x^{\prime},y)\in \Omega  \text{ and } x^{\prime}\equiv x \!\!\mod p^{n-\gamma_{\Omega}-1}  \
\}\]  is  $p$-homogenous with  the branched level set $n-1-\{i_1,i_2 \cdots,i_s\}$.

%For $0\leq l\leq n-1$, define  the map $M_\ell: \Z_{p^n}\times \Z_p\to \Z_{p^{\ell}}\times  \Z_p$ by
%\[M_\ell (x,y)= (x \!\!\!\!\mod p^\ell,y).\]
%Consider the map $M=M_{n-1-\gamma_{\Omega}}$.

%The condition $(p^{i_{t-1}},0)\in  \mathcal{Z}_{\Omega}$ implies that
%
%For $ (p^{\gamma_{\Omega}}, \alpha)$,   define a map $\pi_{(p^{\gamma_{\Omega}}, \alpha)}: \Z_{p^n} \times \Z_p \to \Z_{p^n}$ by
%\[ \pi_{(p^{\gamma_{\Omega}}, \alpha)}(x,y)=p^{\gamma_{\Omega}}x+ \alpha y p^{n-1}.\]
%Note that if $x,y \in  \Omega^{s}_{b,\ell}$ for  some $b$, then $\pi_{(p^{\gamma_{\Omega}}, \alpha)}(x,b)=\pi_{(p^{\gamma_{\Omega}}, \alpha)}(y,b)$. The condition $(p^{\gamma_{\Omega}}, \alpha) \in \mathcal{Z}_{\Omega}$ implies that
%
%Define  the map $\pi_{1}^\ell: \Z_{p^n}\times \Z_p\to \Z_{p^{\ell}}$ by
%\[\pi_{1}^\ell(x,y)= x \!\!\!\pmod{p^\ell}.\]
%The condition $(p^{i_{t-1}},0)\in  \mathcal{Z}_{\Omega}$ implies that
\end{proof}

Theorem  \ref{main2} follows from Propositions \ref{case t-1}, \ref{prop4.6}, \ref{prop4.7} and Lemma \ref{lem4.9}.

\begin{figure}[h]
  	\centering
 	\includegraphics[width=1\linewidth]{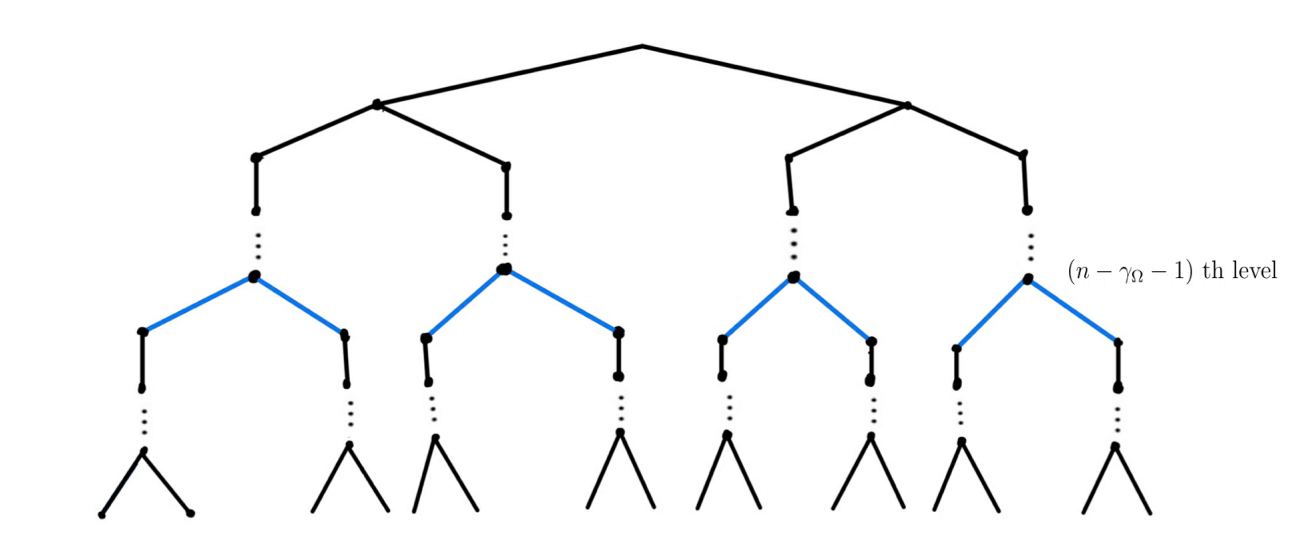}
  	\caption{Corresponding tree of $\mathcal{T}_{\widetilde{\Omega}}$, where the blue edges determined by the hyperplanes ${H}((p^{\gamma_{\Omega}},\alpha), k)$.}
  	\label{fig:5}
 \end{figure}

\bigskip

\bibliographystyle{abbrv}
\bibliography{Ref}

\end{document}